\documentclass{article}
\usepackage{graphicx} 
\usepackage{mathtools}
\usepackage{amsmath}
\usepackage{amssymb}
\usepackage{amsfonts}
\usepackage{color}
\usepackage{lineno}
\usepackage{caption}
\usepackage {indentfirst}
\usepackage{graphicx}
\numberwithin{equation}{section}
\numberwithin{table}{subsection}
\usepackage{algorithm}
\usepackage{algorithmic,float}
\usepackage{multirow}
\usepackage{color}

\floatname{algorithm}{Algorithm}

\numberwithin{algorithm}{section}
\newtheorem{theorem}{Theorem}[section]
\newtheorem{lemma}[theorem]{Lemma}    
\newtheorem{Definition}[theorem]{Definition}

\newenvironment{proof}{{\noindent\it Proof}\quad}{\hfill $\square$\par}

\usepackage{algorithmic,float}

\makeatletter

\usepackage{geometry}
\usepackage[marginal]{footmisc}

\begin{document}

\title{A preconditioned iteration method for solving saddle point problems\footnote{\noindent
  }}
 \author{Juan Zhang\thanks{Corresponding author: zhangjuan@xtu.edu.cn}, Yiyi Luo\thanks{Email: l19118069216@163.com}\\Hunan Key Laboratory for Computation and Simulation in Science and Engineering, \\Key Laboratory of Intelligent Computing and Information Processing of Ministry of Education,\\School of Mathematics and Computational Science, Xiangtan University}
\date{}

\maketitle

\noindent	{\small{\bf Abstract.}
	 This paper introduces a preconditioned method designed to comprehensively address the saddle point system with the aim of improving convergence efficiency. In the preprocessor construction phase, a technical approach for solving the approximate inverse matrix of sparse matrices is presented. The effectiveness of the proposed method is demonstrated through numerical examples, emphasizing its efficacy in approximating the inverse matrix. Furthermore, the preprocessing technology includes a low-rank processing step, effectively reducing algorithmic complexity. Numerical experiments validate the effectiveness and feasibility of PSLR-GMRES in solving the saddle point system.
		
\noindent	{\small{\bf Keywords.}
	Saddle Point problems, Preprocessing, Arnoldi algorithm, Low rank correction technique, Approximate inverse matrix of Schur complement.

\section{Introduction}
Consider the following general saddle point linear system 
\begin{equation}\label{eq01}
	\begin{bmatrix}
		A^TA&B^T\\
		\pm B&C
	\end{bmatrix}\begin{bmatrix}
		x\\y
	\end{bmatrix}=\begin{bmatrix}
		f\\g
	\end{bmatrix},\end{equation} where $A\in \mathbb{R}^{m\times n}$ is a non-singular matrix, that is, $A^TA$ is non-singular\cite{ref16},$B\in \mathbb{R}^{m\times n}$, $x,f\in \mathbb{R}^n$, and $y,g\in \mathbb{R}^m$ with $m\leqslant n$, $C\in \mathbb{R}^{m\times m}$ is symmetrically definite. Such structured linear equation systems are widely used in many practical problems, such as quadratic optimization problems with equality constraints, computational fluid dynamics, and meshless methods. When dealing with a large-scale saddle-point problem \eqref{eq01}, iterative methods, such as Uzawa method\cite{ref1}, HSS iterative method\cite{ref2}, and Preconditioned Conjugate Residuals (PCR) method\cite{ref3}, Preconditioned Conjugate Gradient (PCG) method\cite{ref38, ref43} and so on are commonly employed for its solution.
 
However, the linear system encounters issues such as slow convergence speed and long CPU solving time. To address these problems, a good preconditioner is typically constructed based on specific matrix structures to reduce computational workload and accelerate convergence. The basic idea of preconditioning is to transform the coefficient matrix of the original linear system into a more easily solvable linear system, thereby accelerating the solution process. Preconditioning methods are typically used in iterative algorithms to improve their convergence performance by applying auxiliary operations at each iteration step. Currently, there are several methods for preconditioning, including incomplete LU decomposition preconditioning (ILU preconditioning)\cite{ref34,ref42}, incomplete Cholesky decomposition (ICF)\cite{ref40,ref41}, block preconditioning\cite{ref6},  etc. As researchers delve deeper into this field, variations of these preconditioning methods have also emerged. For example, block diagonal preconditioning involves using a block diagonal matrix as the preconditioner, and scaling parameters are added to the block matrix to adjust the stability of the preconditioner, resulting in faster convergence speed \cite{ref20}. When solving the KKT form of saddle-point systems, the SLRU method \cite{ref22} is used to obtain an approximate inverse matrix of $A$. Then, the Sherman-Morrison-Woodbury formula \cite{ref12} is applied specifically to solve for the inverse of the Schur Complement $S^{-1}$
  \cite{ref21}. When preprocessing saddle point systems, it is crucial to construct an approximate Schur complement matrix or an approximate inverse of the Schur complement matrix. Obtaining a good approximation of the Schur complement matrix will enable the saddle point system's preprocessor to achieve better performance and more accurate convergence speed. As a result, an increasing number of methods have emerged to address this issue. There are many methods for constructing approximate Schur complements. For example, there is an efficient technique for obtaining a high-quality approximation of the Schur complement matrix in the preconditioning method for finite element discretization of elliptic boundary value problems. The Schur complement is based on a two-by-two block matrix decomposition and is computed through assembling local (macro-element) Schur complements to enhance numerical efficiency\cite{ref23}. Furthermore, there is a high-quality sparse approximation method for constructing second-order block-structured matrices under finite element discretization, especially for the Schur complement of matrices obtained from the discretization of partial differential equations\cite{ref26}. Additionally, a preconditioning method, known as the element-by-element approximate Schur complement technique, has been developed for the grid-based non-discrete piezoelectric equation. It has been proven that this method can accelerate the convergence of Krylov subspace iterative methods, thereby significantly improving the efficiency of numerical calculations\cite{ref27}.
Although these methods have been well developed, there is little mention of the method of solving the inverse of $S$ so one of the motivations of this paper is to propose an approximate method for solving the inverse matrix of the Schur complement.

In addition to the aforementioned commonly used methods, many other preprocessing methods have gradually emerged. 
Xiaojun Chen and Kouji Hashimoto proposed and analyzed a simple method in 2003 to numerically verify the accuracy of approximate solutions to saddle point matrix equations, which only requires the iterative solutions of two symmetric positive definite linear systems .
Michele Benzi and Gene Howard Golub investigated the problem of solving saddle point linear systems using preconditioned Krylov subspace methods in 2004. They proposed a preconditioning strategy based on the symmetric/skew-symmetric splitting of the coefficient matrix and established useful properties of the preconditioning matrix .
Yang Cao, Zhiru Ren, and Linquan Yao developed an improved Relaxed Positive Semi-Definite and Hermitian Splitting (IRPSS) preconditioner for solving saddle point problems in 2019. These preconditioners are easier to implement than the Relaxed Positive Semi-Definite and Hermitian Splitting (RPSS) preconditioners \cite{ref10}.
Rui Li and Zeng Qi Wang proposed a restricted preconditioned conjugate gradient method for solving saddle point systems in 2021 \cite{ref4}.
Nana Wang and Jicheng Li introduced an exact parameterized block symmetric positive definite preconditioner for solving $3\times3$ block saddle point problems and its inexact preconditioned version in 2022. They also analyzed their eigenvalues and investigated the selection of (approximate) optimal parameters in the aforementioned inexact preconditioners \cite{ref11}. However, these methods have their limitations. For instance, some methods may be dependent on the choice of constraints when analyzing convergence properties and computing solution times. Additionally, the effectiveness of specific preconditioning techniques may vary depending on the characteristics of the problem, requiring a strong background in specialized knowledge. Therefore, these methods are not very suitable for our saddle point system.

With the development of saddle point matrix preprocessing, many researchers have studied approximate inverse preconditioners based on low-rank approximations. For instance, Ruipeng Li, Yuanzhe Xi, and Yousef Saad introduced the Schur Complement Low-Rank (SLR) preconditioner in \cite{ref15}. Its objective is to minimize the rank difference between the Schur complement and the inverse of the submatrix equation \eqref{PSLRflow}. Multi-level extensions of SLR, such as Multi-Level Single Reverse (MSLR) \cite{ref18} and Generalized MSLR (GMSLR) \cite{ref19}, address the issue of the Schur complement itself being potentially very large for 3D problems. They exploit the block-diagonal structure of the Schur complement to further reduce computational costs. In MSLR and GMSLR preconditioners, hierarchical interface decomposition (HID) is performed using nested dissection algorithms.  These preconditioners utilize different low-rank corrections to approximate the Schur complement or its inverse. These methods, which can be considered as approximate inverse techniques, are often effective for solving indefinite linear systems. However, they also face certain challenges. For instance, controlling the size of the interface during the construction phase is difficult when employing nested dissection.

 One of the main objectives of this paper is to propose an alternative approach that avoids nested dissection. We introduce a method that combines a low-rank approximation technique with a simple Neumann polynomial expansion technique (\cite{ref39}, Section 12.3.1). This approach, known as the Power Series Low-Rank Correction (PSLR) preconditioner, aims to enhance robustness. Applying Neumann polynomial preprocessing techniques directly to the Schur complement 
$S$  involves approximating 
$(\omega S)^{-1}$
  through polynomial expansion of the 
$m$ term, as detailed in section 12.3.1 of \cite{ref39}. This can be expressed as
\begin{equation}\label{eq:classic_power_series}
\frac{1}{\omega}\left[I+N+N^2+\cdots+N^m\right]D^{-1},
\end{equation}
where $\omega$  represents a scaling parameter, 
$D$ denotes the (block) diagonal of $S$. However, the approach presented in Equation \eqref{eq:classic_power_series} also exhibits certain limitations. For instance, determining an optimal value for the parameter $\omega$ can be challenging. Additionally, the convergence of the matrix series in  \eqref{eq:classic_power_series} is only ensured when $\rho(N) < 1$, and the accuracy of the approximation improves with increasing $m$ only under specific conditions that may not apply to general matrices. Furthermore, even when 
$\rho(N) < 1$, \eqref{eq:classic_power_series} offers only a coarse approximation of 
$S^{-1}$ for small values of $m$, while utilizing a large $m$ can lead to significant computational cost.

The PSLR preconditioner seamlessly combines power series expansion with certain low-rank correction techniques to overcome these drawbacks. On one hand, this preconditioning technique enhances robustness. When $\rho(N)>1$, the classical Neumann series defined by \eqref{eq:classic_power_series} diverges, and the approximation accuracy deteriorates as $m$ increases. However, this issue can be addressed by employing low-rank correction techniques. Specifically, we utilize low-rank correction as a form of deflation to shift the eigenvalues of $N$ with modulus greater than 1 closer to zero. The objective is to ensure the convergence of the series \eqref{eq:classic_power_series} towards the inverse matrix of the Schur complement. On the other hand, this method enhances decay properties. The performance of previously developed methods, such as SLR, MSLR, and GMSLR, relies on the decay properties of eigenvalues associated with the inverse of the Schur complement, $S^{-1}$. If the decay rate is slow, these preconditioners become ineffective. Furthermore, the PSLR preconditioner allows for controlling the decay rate of matrix eigenvalues by adjusting the number of terms $m$ in the power series expansion. This can significantly improve performance. Additionally, PSLR is effective in handling general sparse problems. Unlike SLR and MSLR, it is not limited to symmetric systems. Numerical experiments demonstrate that the PSLR preconditioner outperforms other methods for solving linear equation systems in various tests.

The structure of this paper is as follows. First, in Section \ref{solveSchur}, we introduce the Power Series Schur Complement Low-Rank (PSLR) method and the specific approach for solving saddle point matrices. Then, in Section \ref{sec
	:converge}, we analyze the convergence and complexity of the PSLR method with respect to the power series expansion. Finally, in Section \ref{sec:experiments}, we perform tests on the algorithm.

\section{Overview of PSLR pretreatment methods}\label{solveSchur}
The PSLR (Power Schur Complement Low-rank) \cite{ref6} preconditioning depends on the structure of the matrix, and requires the use of the approximate inverse of the saddle point matrix after a congruent transformation. To better distinguish the form of the saddle point matrix, we make the following definition.

\begin{Definition}
	Let $A\in \mathbb{R}^{m\times n}$ is a non-singular matrix, $B\in \mathbb{R}^{m\times n}$, $C\in \mathbb{R}^{m\times m}$is a singular matrix, then the matrix $\begin{bmatrix}
		A^TA&B^T\\-B&C\end{bmatrix}$ is called a positive saddle point matrix, and the matrix $\begin{bmatrix}
		A^TA&B^T\\B&-C
	\end{bmatrix}$ is called a negative saddle point matrix.
\end{Definition}

\cite{ref6}We first perform elementary transformations on the positive saddle point matrix (the same applies to the negative saddle point matrix) to simplify the structure of the matrix

\begin{equation}
	\overbrace{\begin{bmatrix}
			I&0\\B(A^TA)^{-1}&I
	\end{bmatrix}}^{\mathcal{P}}\overbrace{\begin{bmatrix}
			A^TA&B^T\\-B&C
	\end{bmatrix}}^{\mathcal{A}}\overbrace{\begin{bmatrix}
			I&-(A^TA)^{-1}B^T\\0&I
	\end{bmatrix}}^{\mathcal{Q}}
	=\overbrace{\begin{bmatrix}
			A^TA&0\\0&S\end{bmatrix}}^{\mathcal{H}}.\end{equation}Where $S=C+B(A^TA)^{-1}B^T$ is a negative Schur complement of $\mathcal{A}$.

Then the coefficient matrix can be converted to
$\mathcal{H}=\mathcal{P}\mathcal{A}\mathcal{Q}.$ The saddle point system can be converted into an equivalent linear system,\begin{equation}\mathcal{P}\mathcal{A}\mathcal{Q}\mathcal{Q}^{-1}z=\mathcal{P}b,\end{equation} 
where $\mathcal{P},\mathcal{Q}$ is the upper and lower triangular matrix with diagonal element 1, and b is the vector at the right end of the linear system.

If $\mathcal{Q}^{-1}z=\hat{z},\mathcal{P}b=\hat{b},\mathcal{P}\mathcal{A}\mathcal{Q}=\hat{\mathcal{A}}$,
we have $\hat{\mathcal{A}}\hat{z}=\hat{b}$, namely
\begin{equation}\label{eq10}\begin{bmatrix}
		A^TA&0\\0&S
	\end{bmatrix}\begin{bmatrix}
		x+(A^TA)^{-1}B^Ty\\y
	\end{bmatrix}=\begin{bmatrix}
		f\\B(A^TA)^{-1}g
	\end{bmatrix}.\end{equation}

For the expansion of equation \eqref{eq10}, there are\begin{equation}\label{PSLRflow}
	\begin{cases}
		A^TAx+B^Ty=f\\
		Sy=B(A^TA)^{-1}g
	\end{cases}. \end{equation}
In equation \eqref{PSLRflow}, the second equation can be solved by first solving for $y$ and then substituting into the first equation. The solution for $y$ is given by 
\begin{equation}\label{eq12}
	y=S^{-1}B(A^TA)^{-1}g.
\end{equation} The key to solving the system of linear equations is to compute Schur's approximate inverse.

The power series expansion presented in this section applies to matrices split into the sum (difference) of a unit matrix and another matrix, rather than to the matrix itself. Since there are many forms of saddle point systems, without loss of generality, we only discuss the positive saddle point system in this section. 

First, the Schur complement $S$ can be written as the sum of two  different matrices 
$$
S=C+B(A^TA)^{-1}B^T=C(I+C^{-1}B(A^TA)^{-1}B^T).$$
If there exists a real matrix $A$, then $$(I-A)^{-1}=I+A+A^2+\cdots=\sum_{i=0}^{\infty}A^i.$$ 
Therefore$$S^{-1}=(I+C^{-1}B(A^TA)^{-1}B^T)^{-1}C^{-1}=\sum_{i=0}^{\infty}(-C^{-1}B(A^TA)^{-1}B^T)^iC^{-1}.$$ Thus, we can transform the second equation of \eqref{PSLRflow} into

\begin{equation}\label{eq11}\begin{aligned}
		y&=S^{-1}B(A^TA)^{-1}g\\
		&=(I+C^{-1}B(A^TA)^{-1}B^T)^{-1}C^{-1}B(A^TA)^{-1}g\\
		&=\sum_{i=0}^{\infty}(-C^{-1}B(A^TA)^{-1}B^T)^iC^{-1}B(A^TA)^{-1}g\\
		&\approx\sum_{i=0}^{m}(-1)^i(C^{-1}B(A^TA)^{-1}B^T)^iC^{-1}B(A^TA)^{-1}g,
\end{aligned}\end{equation}where $m$ is the specific number of terms in the power series expansion.

At this stage, we can obtain the approximate inverse matrix of matrix $S$ by performing multiplication calculations between matrices $A$ and $B$, as well as with the column vector $g$. We can then use this approximate value as the initial value for the first step of iteration, and subsequently apply the GMRES method to achieve a faster convergent numerical solution.

From this, it can be seen from Equation \eqref{eq11} that to solve for $y$, it is necessary to determine the number of terms in the power series expansion of the inverse matrix of $S$. In \eqref{eq11}, we can use the method of solving the approximate inverse matrix mentioned above to solve the saddle point system.

Let $M=-C^{-1}B(A^TA)^{-1}B^T$, $S=C(I-M)$. In the following text, $M$ is the same. then
\begin{equation}\label{eq002}
	S^{-1}=\sum_{i=0}^{\infty}M^{i}C^{-1}=\sum_{i=0}^{m}M^iC^{-1}+\sum_{i=m+1}^{\infty}M^iC^{-1}.
\end{equation}
Multiplying both sides of the equation by $S$, we have $$I=S\sum_{i=0}^{m}M^iC^{-1}+S\sum_{i=m+1}^{\infty}M^iC^{-1}.$$
If $E_{rr}(m)=S\sum_{i=m+1}^{\infty}M^iC^{-1}$, we can get \begin{equation}\begin{aligned}\label{eq004}
		E_{rr}(m)=&I-[S\sum_{i=0}^{m}M^iC^{-1}]\\
		=&I-C(I-M)(\sum_{i=0}^{m}M^i)\\
		=&I-C(\sum_{i=0}^{m}M^iC^{-1}-\sum_{i=0}^{m}M^{i+1}C^{-1})\\
		=&I-C(C^{-1}-M^{m+1}C^{-1})\\
		=&CM^{m+1}C^{-1}.
\end{aligned}\end{equation}
Therefore, we have \begin{equation}\label{eq003}
	S^{-1}=\sum_{i=0}^{m}M^iC^{-1}(I-E_{rr}(m))^{-1}=
	\sum_{i=0}^{m}M^iC^{-1}(I-CM^{m+1}C^{-1})^{-1}.
\end{equation}

\subsection{Low-rank Matrix Correction}\label{lowrank}
Next, we need to use the Arnoldi algorithm to approximate $E_{rr}(m)$. For the Arnoldi algorithm, we can achieve low-rank matrix by controlling the order of $H_{r_k}$. Let $V_{r_k}\in \mathbb{R}^{n\times r_k}$ and $H_{r_k}\in \mathbb{R}^{r_k\times r_k}$, then we have $$H_{r_k}=V_{r_k}^TAV_{r_k}.$$

At this point, $H_{r_k}$ is an $r_k$-order matrix $\rm{(}r_k\ll n\rm{)}$, and we have reduced the order of the matrix.Using the Arnoldi algorithm, we have $E_{rr}(m)\approx V_{r_{k}}H_{r_{k}}V_{r_{k}}^T.$
Then we can get \begin{equation}\label{002}\begin{aligned}
		S^{-1}&\approx[\sum_{i=0}^{m}M^iC^{-1}](I-V_{r_{k}}H_{r_{k}}V_{r_{k}}^T)^{-1}\\
		&=[\sum_{i=0}^{m}M^iC^{-1}](V_{r_{k}}V_{r_{k}}^T-V_{r_{k}}H_{r_{k}}V_{r_{k}}^T)^{-1}\\
		&=[\sum_{i=0}^{m_1}M^iC^{-1}](V_{r_{k}}(I-H_{r_{k}})^{-1}V_{r_{k}}^T)\\
		&=[\sum_{i=0}^{m}M^iC^{-1}](I+V_{r_{k}}(I-H_{r_{k}})^{-1}V_{r_{k}}^T-I)\\		
		&=[\sum_{i=0}^{m}M^iC^{-1}](I+V_{r_{k}}(I-H_{r_{k}})^{-1}V_{r_{k}}^T-V_{r_{k}}V_{r_{k}}^T)\\
		&=[\sum_{i=0}^{m}M^iC^{-1}](I+V_{r_{k}}[(I-H_{r_{k}})^{-1}-I]V_{r_{k}}^T)\\
		&=[\sum_{i=0}^{m}M^iC^{-1}](I+V_{r_{k}}G_{r_{k}}V_{r_{k}}^T).
\end{aligned}\end{equation}
where $G_{r_{k}}=(I-H_{r_k})^{-1}-I$.

According to the above analysis, when $E_{rr}(m)$ is determined, i.e., the number of terms of the power series expansion is determined, the saddle point linear system is also determined.

However, when performing the correction of low-rank matrices, the algorithm directly truncates the matrix from a certain row or column, which results in a large approximation error. Therefore, we use the matrix $\left\|V_{r_k}H_{r_k}V_{r_k}^T-E_{rr}(m)\right\|$ to judge the quality of the approximation.

In addition, due to some precision loss in Arnoldi process, it is also impossible to obtain an accurate solution when solving $S^{-1}$. Next, we will study the factors that affect the accuracy of $S^{-1}$.

\begin{theorem}
	For any matrix norm $\|\cdot\|$, the approximate accuracy of the $S^{-1}$ application program satisfies the following inequality,
	\begin{equation}
		\frac{\left\|S^{-1}-S_{app}^{-1}\right\|}{\left\|S^{-1}\right\|}
		\leqslant\left\|X(m,r_k)\right\|\left\|Z(r_k)^{-1}\right\|,\end{equation}
	where
	$$X(m,r_k)=E_{rr}(m)-V_{r_k}H_{r_k}V_{r_k}^T,Z(r_k)=I-V_{r_k}H_{r_k}V_{r_k}^T.$$
\end{theorem}

\begin{proof}Due to\begin{equation}S^{-1}=\left[\sum_{i=0}^mM^iC^{-1}\right](I-E_{rr}(m))^{-1},\end{equation}
	\begin{equation}S_{app}^{-1}=\left[\sum_{i=0}^mM^iC^{-1}\right](I-V_{r_k}H_{r_k}V_{r_k}^T)^{-1},\end{equation}
	where $M=-C^{-1}B(A^TA)^{-1}B^T$,
	we have \begin{equation}\begin{aligned}
		S^{-1}-S_{app}^{-1}&=\left[\sum_{i=0}^mM^iC^{-1}\right]\left[(I-E_{rr}(m))^{-1}-(I-V_{r_k}H_{r_k}V_{r_k}^T)^{-1}\right]\\
		&=\left[\sum_{i=0}^mM^iC^{-1}\right]\left[(Z(r_k)-X(m,r_k))^{-1}-Z(r_k)^{-1}\right]\\
		&=\left[\sum_{i=0}^mM^iC^{-1}\right](Z(r_k)-X(m,r_k))^{-1}\left[I-(Z(r_k)-X(m,r_k))Z(r_k)^{-1}\right]\\
		&=\left[\sum_{i=0}^mM^iC^{-1}\right](Z(r_k)-X(m,r_k))^{-1}(Z(r_k)-(Z(r_k)-
		X(m,r_k)))Z(r_k)^{-1}\\
		&=\left[\sum_{i=0}^mM^iC^{-1}\right](I-E_{rr}(m))^{-1}X(m,r_k)Z(r_k)^{-1}\\
		&=S^{-1}X(m,r_k)Z(r_k)^{-1}.
	\end{aligned}\end{equation}
	Taking the matrix norm on both sides of the equation, we have
	\begin{equation}\left\|S^{-1}-S_{app}^{-1}\right\|\leqslant\left\|S^{-1}\right\|\left\|X(m,r_k)\right\|\left\|Z(r_k)^{-1}\right\|.\end{equation}
	
	From the above theorem, we can see that the approximate accuracy of $S^{-1}$ is related to $X(m,r_k)$ and $Z(r_k)$, i.e., They affect the accuracy of the solution. 
\end{proof}

Thus, we study the distribution of eigenvalues to demonstrate the difference between performing low-rank processing and not performing low-rank processing. This is done to illustrate the necessity of low-rank processing. We explore the relationship between the eigenvalues of matrices expanded with same numbers of terms in the power series and the value 1. We found a sparse symmetric matrix $A^TA$ of order 494 in the matrix market to construct the upper-left block of the saddle point matrix. The remaining block matrix $B$ is a 494-order random matrix used to construct the upper-right and lower-left matrix blocks. The lower-right elements are the identity matrix of the same order as $B$. When the power series is expanded to 5 terms, the eigenvalue distribution maps generated using low-rank and non-low-rank techniques are as shown in Fig. \ref{fig1}.

\begin{figure}[H]
	\centering
	\includegraphics[width=1\linewidth]{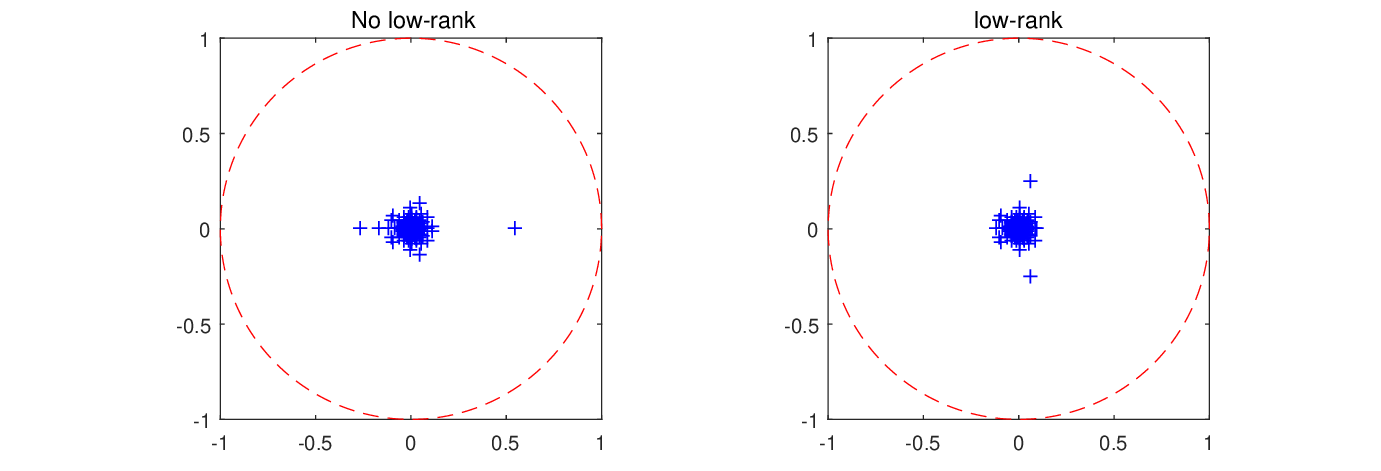}
	\caption{whether or no low-rank processed eigenvalue plots}
 \label{fig1}
\end{figure}

It can be seen that the eigenvalue distribution of the processed matrix with low-rank approximation shows little difference from the Fig. \ref{fig1}. In addition, most of the eigenvalues with modulus less than 1 still cluster near the origin. This indicates that performing low-rank processing does not significantly alter the distribution of matrix eigenvalues. Moreover, during the process of matrix inversion, the order of the matrix is reduced from $n$ to $r_k$, reducing the computational cost.

In the process of correcting low-rank matrices, the selection of the number of terms in the matrix power series expansion is also very important. This can be seen from the process of establishing the PSLR preconditioner. Moreover, in theory, different numbers of terms in the power series expansion correspond to different eigenvalue distributions. Next, we provide a numerical experiment to illustrate the above points. Both tests use the Frobenius norm to measure the degree of approximation of the Arnoldi algorithm. A 128-order random matrix was used in the numerical experiment. In the test, we fixed $r_k=15$ and changed the number of terms used in the power series expansion from $m=3$ to $m=5$. The results are shown in the following Fig. \ref{fig2}.
\newpage
\begin{figure}[H]
	\includegraphics[width=1.0\linewidth]{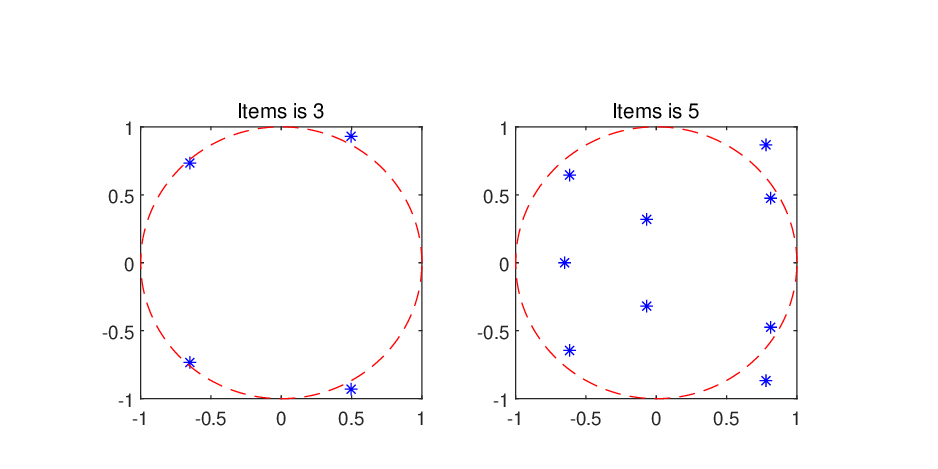}
	\caption{Graph of the change of matrix eigenvalues with the number of unfolded items}
 \label{fig2}
\end{figure}

From the Fig. \ref{fig2}, it can be seen that as the number of terms in the matrix expansion increases, the eigenvalues are more concentrated around the origin, that is, there are more zero eigenvalues. This can effectively reduce the impact of eigenvalues greater than 1 on the convergence of the power series. It can be inferred that when the number of expanded terms $m$ is large enough, there will be more zero eigenvalues in the matrix, and the approximation effect will be better. Therefore, the computational accuracy of the linear system is related to the number of expanded terms. However, we cannot increase the number of power series expansion terms infinitely, because this will increase the computational complexity. Therefore, we will compare the complexity of the PSLR algorithm with the GMSLR algorithm to determine the optimal number of power series expansion terms in Section \ref{sec:com}.

\subsection{Error Analysis of Approximate Inverse Matrix}
Firstly, let us delve into the error involved in solving the approximate inverse matrix of the 
$S$ matrix using the PSLR method. We perform incomplete LU decomposition (ILU) on $A^TA$ in the $S$ matrix, and in this case $(A^TA)^{-1}=U^{-1}L^{-1}$.

Since incomplete LU decomposition has restrictions on the positions of zero elements, namely, the positions and number of zero elements in $A^TA$ are forced to remain the same after ILU decomposition. Therefore, the method for solving the inverse matrix can also be called a method for solving the approximate inverse matrix. In the following analysis of the approximate inverse of the Schur complement matrix, $U^{-1}L^{-1}$ is used to replace $(A^TA)^{-1}$.

From \eqref{eq002}, we know that \begin{equation}S^{-1}=
	\sum_{i=0}^{m}M^iC^{-1}+\sum_{i=m+1}^{\infty}M^iC^{-1}\triangleq\sum_{i=0}^{m}M^iC^{-1}+R,\end{equation}where $M=-C^{-1}B(A^TA)^{-1}B^T$.
Let's analyze the error $R=\sum_{i=m+1}^{\infty}M^iC^{-1}.$ Taking the norm on both sides,\begin{equation}\begin{aligned}
		\left\|R\right\|=\left\|\sum_{i=m+1}^{\infty}M^iC^{-1}\right\|
		\leqslant\sum_{i=m+1}^{\infty}\left\|M\right\|^i\left\|C^{-1}\right\|
		\leqslant\frac{\left\| M\right\|^{m+1}\left\|C\right\|^{-1}}{1-\left\|M\right\|}.
\end{aligned}\end{equation}

We have the following conclusion for $\left\|M\right\|^{m+1}$.
When the spectral radius of $M$ is less than 1, $\frac{\left\|M\right\|^{m+1}}{1-\left\|M\right\|}$ converges and is bounded. Therefore, we can set a limit for the error and use a finite number of terms in the series to approximate the Schur complement $S$'s approximate inverse.

When the spectral radius of $M$ is greater than 1, we use the methods described in the Section \ref{lowrank},
define $\hat{R}=\sum_{i=m+1}^{\infty}M^i$, and in this case, there is no restriction on $M$ having $\rho(M)<1$. In particular, when $\rho(M)<1$, $\hat{R}=R$.

By the definition of a power series, we can split it into two terms and let
\begin{equation}S^{-1}=\sum_{i=0}^{m}M^iC^{-1}+\sum_{i=m+1}^{\infty}M^iC^{-1}
	\triangleq\sum_{i=0}^{m}M^iC^{-1}+\hat{R}.\end{equation}where $M=-C^{-1}B(A^TA)^{-1}B^T$.

Multiplying both sides of the equation by $S$, we get $$I=S\sum_{i=0}^{m}M^iC^{-1}+S\hat{R}.$$
Let $E_{rr}(m)=S\hat{R}=S\sum_{i=m+1}^{\infty}M^iC^{-1}$. And from \eqref{eq004}, we have 
$$E_{rr}(m)=CM^{m+1}C^{-1}.$$
Therefore, it can be seen from the subsection \ref{solveSchur}, we have \begin{equation}\begin{aligned}
		S^{-1}=
  \left[\sum_{i=0}^{m}M^iC^{-1}\right](I-E_{rr}(m))^{-1}= \left[\sum_{i=0}^{m}M^iC^{-1}\right](I-CM^{m+1}C^{-1})^{-1}
	\end{aligned}
	.\end{equation}
Using the Arnoldi algorithm, we have $E_{rr}(m)\approx V_{r_{k}}H_{r_{k}}V_{r_{k}}^T.$, where $V_{r_k}$ is a column orthogonal matrix.
In summary, the final expression of the $S$ is \begin{equation}S^{-1}\approx[\sum_{i=0}^{m}M^iC^{-1}](I+V_{r_{k}}G_{r_{k}}V_{r_{k}}^T).
\end{equation}
where $G_{r_{k}}=(I-H_{r_k})^{-1}-I$.

From the above analysis, we can see that when $E_{rr}(m)$ is determined, i.e., the number of terms in the power series expansion is determined, the saddle point linear system is also determined. At this point, the power series can be transformed into a finite sum, and there is no approximation error.

\subsection{Pinv method}\label{A=I-Fmethod}

Based on the aforementioned approach, we have derived a method for solving the linear equation system 
$Ax=b$, which we will refer to as the Pinv method, as presented in Theorem \ref{approximate inverse formula}.

\begin{theorem}\label{approximate inverse formula}
	For any matrix $A \in \mathbb{R}^{n \times n}$, and $A=I-F$, $F\in\mathbb{R}^{n \times n}$ is the matrix from which $A$ splits, then an approximate inverse of $A$ is given by \begin{equation}
		A^{-1}=\left(\sum_{i=0}^{m}F^{i}\right)(I-F^{m+1})^{-1}=(\sum_{i=0}^{m}F^{i})(I+V_{r_k}G_{r_k}V_{r_k}^T),
	\end{equation}where  $$F^{m+1}=V_{r_k}H_{r_k}V_{r_k}^T,V_{r_k}\in \mathbb{R}^{n\times r_k}(r_k\ll n),~H_{r_k}\in \mathbb{R}^{r_k\times r_k},~G_{r_k}=(I-H_{r_k})^{-1}-I.$$ And we can express the solution for $x$ as \begin{equation} x=A^{-1}b=(I-F)^{-1}b=\left(\sum_{i=0}^{m}F^{i}\right)(I+V_{r_k}G_{r_k}V_{r_k}^T)b.
	\end{equation}
\end{theorem}

\begin{proof}
	Let $A=I-F,$then $$x=A^{-1}b=(I-F)^{-1}b=\sum_{i=0}^{\infty}F^ib.$$ So the key is that we need to find a way to solve $\sum_{i=0}^{\infty}F^i$. The main idea is to deal with the series expansion of an infinite term so that it becomes a finite term series. The result of the processing is \begin{equation}\begin{aligned}	I&=A\sum_{i=0}^{\infty}F^i\\&=A\sum_{i=0}^{m}F^{i}+A\sum_{i=m+1}^{\infty}F^i\\
			&=A\sum_{i=0}^{m}F^{i}+(I-F)\sum_{i=m+1}^{\infty}F^i\\
			&=A\sum_{i=0}^{m}F^{i}+\sum_{i=m+ 1}^{\infty}F^i-\sum_{i=m+1}^{\infty}F^{i+1}\\     
			&=A\sum_{i=0}^{m}F^{i}+\sum_{i=m+1}^{\infty}F^i-\sum_{i=m+2}^{\infty}F^{i}\\&=A\sum_{i=0}^{m}F^{i}+F^{m+1}.
	\end{aligned}\end{equation}
	Hence, we can express the approximate inverse of $A$ as \begin{equation}\begin{aligned}
			A^{-1}=\left(\sum_{i=0}^{m}F^{i}\right)(I-F^{m+1})^{-1}.
	\end{aligned}\end{equation}Then we get $x=A^{-1}b=\sum_{i=0}^{m}F^ib. $
	
	For $F^{m+1}$, we can use the Arnoldi method to approximate the matrix as $$F^{m+1}=V_{r_k}H_{r_k}V_{r_k}^T,$$ where $V_{r_k}$ is the column orthogonal matrix and $H_{r_k}$ is the upper Hessenberg matrix.
	
	Therefore, we can write: \begin{equation}\begin{aligned}
			A^{-1}&=\left(\sum_{i=0}^{m}F^{i}\right)(I-F^{m+1})^{-1}\\&=\left(\sum_{i=0}^{m}F^{i}\right)(I-V_{r_k}H_{r_k}V_{r_k}^T)^{-1}\\
			&=\left(\sum_{i=0}^{m}F^{i}\right)(V_{r_k}(I-H_{r_k})^{-1}V_{r_k}^T)\\
			&=\left(\sum_{i=0}^{m}F^{i}\right)(I+V_{r_k}G_{r_k}V_{r_k}^T),
	\end{aligned}\end{equation}
	where $G_{r_k}=(I-H_{r_k})^{-1}-I$.
	
	Finally, we can obtain the solution for the general linear equation system $Ax=b$ as \begin{equation}x=A^{-1}b=\left(\sum_{i=0}^{m}F^{i}\right)(I+V_{r_k}G_{r_k}V_{r_k}^T)b.
	\end{equation} where $G_{r_k}=(I-H_{r_k})^{-1}-I$, $H_{rk}$  is the Hessenberg matrix produced by $F^{m+1}$ in the Arnoldi process.
	
	In this process we did not touch on the condition that the spectral radius of the matrix must be less than 1.
\end{proof}

\section{Analysis of the Result of Power Series Expansion}\label{sec
	:converge}
\subsection{Convergence Analysis of Power Series}

In the preprocessing process, we mainly use direct methods to construct the PSLR preconditioner to solve the linear system's initial solution. We then use the GMRES method to iteratively solve it, ultimately minimizing the error. In the process of constructing the preconditioner, the power series is used to approximate the matrix inverse. Here we give the conditions for the convergence of a power series. To prove the convergence of the power series, we first introduce some lemmas.

\begin{Definition}{\rm \cite{ref16}}
	A matrix $A\in M_n$ is said to be convergent if $\lim_{k\to\infty}A^k=0$.
\end{Definition}

\begin{lemma}{\rm\cite{ref16}}\label{le1}
	For $A\in \mathbb{C}^{n \times n}$, we have
	
	\textup{(1)} For any compatible norm $\left\|\cdot\right\|$ in $\mathbb{C}^{n \times n}$, $\rho(A)\leqslant\left\|A\right\|$.
	
	\textup{(2)} For any given $\epsilon>0$, there exists an operator norm $\left\|\cdot\right\|$ in $\mathbb{C}^{n \times n}$ such that $\left\|A\right\|\leqslant\rho(A)+\epsilon$.
\end{lemma}

\begin{lemma}{\rm\cite{ref16}}
	For $A\in \mathbb{C}^{n \times n}$, $\lim_{k\to\infty}A^k=0$ if and only if $\rho(A)<1$.
\label{le2}\end{lemma}

\begin{lemma}\textup{\cite{ref16}} \label{le3} Let $R$ be the convergence radius of the pure scalar power series $\sum_{k=0}^{\infty}a_kz^k$, and let $A\in M_n$ be given. If $\rho(A)<R$, then the matrix power series $\sum_{k=0}^{\infty}a_kA^k$ converges, where $\rho(A)$ denotes the spectral radius of the matrix $A$.
\end{lemma}

\begin{theorem}
	\textup{\cite{ref12}} For $A\in \mathbb{C}^{n \times n}$, we have
	
	\textup{(1)} The power series $\sum_{k=0}^{\infty}A^k$ converges if and only if $\rho(A)<1$.
	
	\textup{(2)} When $\sum_{k=0}^{\infty}A^k$ converges, we have $$\sum_{k=0}^{\infty}A^k=(I-A)^{-1}.$$ and 	for all positive integers $m$, there exists an operator norm $\left\|\cdot\right\|$ in $\mathbb{C}^{n \times n}$ such that \begin{equation}\left\|(I-A)^{-1}-\sum_{k=0}^{m}A^k\right\|\leqslant\frac{\left\|A\right\|^{m+1}}{1-\left\|A\right\|}.\end{equation} 
	
\end{theorem}

\begin{proof}
	Since $\sum_{k=0}^{\infty}A^k$ converges, we have $\lim_{k\to\infty}A^k=0$. By Lemma \ref{le2}, we know that $\rho(A)<1$. We have thus demonstrated the necessity of the theorem. In addition,
	the convergence radius of the power series $\sum_{k=0}^{\infty}z^k$ is $R=1$. Therefore, According to Lemma \ref{le3}, when $\rho(A)<1$, $\sum_{k=0}^{\infty}A^k$ converges. We have proven the sufficiency of the theorem.
\end{proof}

For power series expansion with a matrix spectral radius greater than 1, the convergence is not guaranteed, and divergent power series cannot provide a sufficiently accurate approximation of the inverse matrix. Therefore, we transform the power series expansion to indirectly solve $A^{-1}$.

As discussed in Section \ref{A=I-Fmethod}, considering the analysis of the matrix spectral radius, we have $$A^{-1}=\left(\sum_{i=0}^{m}F^i\right)(I+V_{r_k}G_{r_k}V_{r_k}^T),$$ where $A=I-F$ and $G_{r_k}=(I-H_{r_k})^{-1}-I$. It is evident that the power series has been transformed into a finite series, which can be computed. Although this method only provides an approximate inverse of the matrix, we can utilize the GMRES method to search for a more accurate solution to the linear system.

\subsection{Complexity Analysis of PSLR-GMRES Method}\label{sec:com}
Based on the above analysis, the number of expanded terms in the power series has a significant impact on the calculation results. Therefore, it is crucial to determine the optimal number of terms for the power series expansion. In this section, we present the general framework of the PSLR-GMRES algorithm and compare the complexity of PSLR, MSLR, and GMSLR algorithms. We perform complexity analysis on the positive saddle point matrix.

The PSLR-GMRES algorithm consists of two main parts. The first part summarizes methods for approximating the inverse of a general matrix $A$. In the second part, a PSLR preprocessor is established to process the Schur complement $S^{-1}=(I+B(A^TA)^{-1}B^T)^{-1}$. In the third part, $z=PSLR(b)$ is calculated. The GMRES algorithm is then used iteratively to solve the equation.
\subsubsection{Establishing the PSLR preprocessor}
To better describe the process of preprocessor establishment, we present the following Algorithm \ref{al1}.
\begin{algorithm}[H]
	\begin{algorithmic}
		\caption{Calculate the $A^{-1}$ on the general matrix}
		\label{al1}
		\STATE	1:\;Split the matrix so that it is equal to the difference between the identity
		matrix and the matrix $F$, i.e. 
      $$A=I-F.$$	
		
		\STATE	2:\;\textbf{For} $j=1,\cdots,n $ \textbf{Do} ($n$ is the number of expanded terms for  given power series)
		\STATE	3:\;Calculate $\sum_{i=0}^{n}F^i$.
		\STATE	4:\;\textbf{end for}	
		\STATE	5:\;Compute $$E_(rr)(n)=F^{n+1}.$$
		\STATE	6:\;Apply the $Arnoldi$ algorithm, we get
		$$(V_{r_k},H_{r_k})=Arnoldi(E_{rr}(n),r_k).$$
		\STATE	7:\;Calculate $G_{r_k}=(I-H_{r_k})^{-1}-I$.
		\STATE	8:\;Get $A^{-1}=(\sum_{i=0}^{n}F^{i}) (I+V_{r_k}G_{r_k}V_{r_k}^T)$.
	\end{algorithmic}
\end{algorithm}

When approximating the inverse of the Schur complement $S$, the calculation steps are similar to those of approximating the inverse of $A$, as illustrated in the following Algorithm \ref{al2}.

\begin{algorithm}[H]
	\begin{algorithmic}
		\caption{Calculate the $S^{-1}$ on the saddle matrix}
		\label{al2}
		
		\STATE	1:\;Matrix $A^TA$ performs IC decomposition and makes $M=-C^{-1}B(A^TA)^{-1}B^T$.
		\STATE	2:\;Diagonal block matrix tiles for $A^{T}A$. In the following process, the elements 
		on the diagonal are generally extracted,
		$$\left[L_{A},L_{A}\right]=IC(A^TA),(A^TA)^{-1}=(L_{A}^TL_{A})^{-1}.$$
		\STATE	3:\;\textbf{For} $j=1,\cdots,m$ \textbf{Do}($m$ is the number of expanded terms for  given power
		series)	
		\STATE	4:\;Calculate $\sum_{i=0}^{m}M^iC^{-1}$.
		\STATE	\;\textbf{end for}	
		\STATE	5:\;Apply the equation \eqref{eq004},$$E_{rr}(m)=CM^{m+1}C^{-1}.$$
		\STATE	6:\;Apply Arnoldi algorithm, we get
		$$(V_{r_k},H_{r_k})=Arnoldi(E_{rr}(m),r_k).$$
		\STATE	7:\;Calculate $G_{r_k}=(I-H_{r_k})^{-1}-I$.
	\end{algorithmic}
\end{algorithm}

When $S^{-1}$ in equation \eqref{eq004} can be approximated, then the PSLR preprocessing is basically completed and we can substitute the equation and then use GRMES to solve it iteratively.

\begin{algorithm}[H]
	\begin{algorithmic}
		\caption{Algorithmic flow for applying PSLR  $z=PSLR(b)$}
		
		\STATE	1:\;Right-end vector $b=\begin{pmatrix}f\\g\end{pmatrix}$.	
		\STATE	2:\;Compute $y=(g-B(A^TA)^{-1})f$.
		\STATE	3:\;Update $y\leftarrow y+V_{r_k}(G_{r_k}(V_{r_k}^Ty)$.
		\STATE	4:\;Compute
		$M=-C^{-1}B(A^TA)^{-1}B^T,y\leftarrow \sum_{i=0}^{m}M^iC^{-1}y$.
		\STATE	5:\;Solving $A^TAx=f-By$.
		\STATE	6:\;$z=\begin{pmatrix}
			x\\y\end{pmatrix}$.
		
	\end{algorithmic}
\end{algorithm}

\subsubsection{Complexity of PSLR-GMRES Method}
Next, we calculate the cost required for the algorithm. Let $A \in \mathbb{R}^{n\times n}$. First, for ILU decomposition, in the case where the fill-in factor is 0, the calculation process of ILU(0) will be simplified because there is no need to consider the influence of additional fill-in elements. Typically, the time complexity of ILU(0) is close to O(nnz).

Additionally, the complexity mainly lies in the Arnoldi process.
Let $A\in \mathbb{R}^{n\times n}, V_{r_k}\in \mathbb{R}^{n\times r_k}, H_{r_k}\in \mathbb{R}^{r_k\times r_k}$, according to the formula $$E_{rr}(m)=CM^{m+1}C^{-1}=V_{r_k}H_{r_k}V_{r_k}^T.$$ Thus, we summarize the steps of the Arnoldi process as follows.

\textbf{Initialization:} Choose the initial vector $v_1$, compute the matrix $Av_1$, normalize it to obtain the vector $v_2$, and construct the initial matrix $V_{r_k}=[v_1,v_2]$.

The iterative steps are as follows.

\textbf{Step 1}. Compute the vector $w=Av_i-\sum_{j=1}^{i}v_jh_{ji}$, where $h_{ji}=v_j^TAv_i$ represents the projection of $v_i$ onto $v_j$ ($i=2\cdots r_k$).

\textbf{Step 2.} For $j=1$ to $i$, compute $h_{ji}=v_j^TAv_i$.

\textbf{Step 3.} Compute $h_{(i+1)i}=|w|$ and normalize it to obtain the vector $v_{i+1}$.

\textbf{Step 4.} Update the $i$th row and the $(i+1)$th column of the matrix $H_{r_k}$ with $h_{ji}$ and $h_{(i+1)i}$, respectively.

\textbf{Output:} Return the matrices $V_{r_k}$ and $H_{r_k}$.

The complexity of\textbf{ Step 1} is $\mathcal{O}(n)$, where $n$ is the dimension of the vector. The number of iterations in \textbf{Step 2} is $r_k-1$. For each iteration, the complexity of computing the vector $w$ is $\mathcal{O}(n)$, the complexity of computing $h_{ji}$ is $\mathcal{O}(n)$, and the complexity of computing $h_{(i+1)i}$ is $\mathcal{O}(n)$. Therefore, the total complexity of \textbf{Step 2} is $\mathcal{O}((r_k-1)n)$. Thus, the complexity is $\mathcal{O}(n+(r_k-1)n)=\mathcal{O}(r_kn)$.

Next, we calculate the complexity of inverting a matrix, For an $n$-dimensional lower triangular matrix $L$, its inverse matrix $L^{-1}$ is also a lower triangular matrix, and it satisfies the following recursive formula.

$$
(L^{-1})_{i,j} = -\frac{1}{L_{i,i}}\sum_{k=i+1}^n L_{i,k}(L^{-1})_{k,j},\quad i<j,
$$
$$
(L^{-1})_{i,i} = \frac{1}{L_{i,i}},\quad i=1,2,\cdots,n.
$$
Since $L$ is a lower triangular matrix, $L_{i,k}=0$ when $k>i$. Therefore, the above recursive formula only needs to calculate the values when $i<j$, which has a complexity of $\mathcal{O}(n^2)$. Thus, the complexity of finding the inverse matrix of an $n$-dimensional lower triangular matrix is $\mathcal{O}(n^2)$.

Finally, the complexity of the GMRES algorithm is $\mathcal{O}(n^2k)$, where $n$ is the dimension of the matrix and $k$ is the number of iterations. Typically, $k$ is much smaller than $n$.

Finally, the complexity of the GMRES algorithm is 
$\mathcal{O}(n^2k)$, where 
$n$ is the dimension of the matrix and 
$k$ is the number of iterations. Typically, 
$k$ is much smaller than 
$n$. In conclusion, the complexity of this algorithm is: $$\mathcal{O}(n^3)+\mathcal{O}(r_kn)+\mathcal{O}(n^2)+\mathcal{O}(n^2k).$$
However, when solving directly, the complexity used is $\mathcal{O}(n^3)$. It reduces the complexity of solving matrix equations.

In reference \cite{ref6}, when comparing the complexity of PSLR preconditioning with GMSLR and MSLR preconditioners for solving general linear systems, the following conclusions were drawn.
For GMSLR preconditioner,
\setlength {\parindent} {2em}
\begin{itemize}
	\item[(a)] when $m=0$, $s>2$, $f(s)>0$, i.e. $\gamma_{GMSLR}>\gamma_{PSLR}$;
	\item[(b)]when $m=1$, $s\geqslant3$, $f(s)>0$, i.e. $\gamma_{GMSLR}>\gamma_{PSLR}$;
	\item[(c)]when $m=2$, $s\geqslant5$, $f(s)>0$, i.e. $\gamma_{GMSLR}>\gamma_{PSLR}$;
	\item[(d)]when $m=3$, $s>8$, $f(s)>0$, i.e. $\gamma_{GMSLR}>\gamma_{PSLR}$;
	\item[(e)]when $m=4$, $s\geqslant20$, $f(s)>0$, i.e. $\gamma_{GMSLR}>\gamma_{PSLR}$;
\end{itemize}
Moreover, for any $s$, if $m\geq 5$, then $f(s) < 0$.

Similarly, we can compare MSLR with PSLR and obtain the following results.

\begin{itemize}
	\item[(a)]when $m=0$, $s>2$, $\gamma_{MSLR}>\gamma_{PSLR}$;
	\item[(b)]when $m=1$, $s\geqslant3$, $\gamma_{MSLR}>\gamma_{PSLR}$;
	\item[(c)]when $m=2$, $s\geqslant6$, $\gamma_{MSLR}>\gamma_{PSLR}$;
	\item[(d)]when $m=3$, $s\geqslant8$, $\gamma_{MSLR}>\gamma_{PSLR}$;
	\item[(e)]for any $s$, if $m\geq 4$, $\gamma_{MSLR}<\gamma_{PSLR}$.
\end{itemize}

In the algorithm, we set $s$ as the number of blocks for matrix partitioning. Based on the above conclusions, we only expand the terms in the algorithm up to 5.

\section{Experimental results}
\label{sec:experiments}

This section presents numerical experiments that demonstrate the efficiency and robustness of the PSLR preregulator. The experiments include both symmetric and asymmetric cases, and are implemented in Matlab using the PSLR preregulator and the matrix equation solver. The preconditioning construction time includes the ILU decomposition of the inverse matrix, as well as the computation of $V_{r_k}$ and $G_{r_k}$. In practical calculations, the right end vector $b$ is randomly selected to ensure that $Ax=b$ is a random vector. Additionally, the initial guess for the Krylov subspace method is always taken as a zero vector.

In order to better evaluate the performance of the PSLR preconditioner, we tested it using saddle point matrices and compared symmetric and asymmetric matrices. For asymmetric saddle point matrices, we compared the influence of different initial values and different orders on PSLR-GMRES. For symmetric saddle point matrices, we compared the impact of different treatments of $(A^TA)^{-1}$ on PSLR-GMRES, the impact of matrix sorting on PSLR-GMRES and the impact of the number of power series expansion terms $m$ as well as the comparison of the solving speed between CG, ADI, Pinv and PSLR-GMRES.

In the remainder of this section, the following symbols will be used.

its: The number of iterations of GMRES or CG required to reduce the initial remaining norm by $10^{-6}$.

F: Indicates that GMRES or CG failed to converge within 500 iterations.

o-t: The time required to reorder the matrix.

p-t: The time required for pre-regulator construction.

i-t: The wall clock time required for the iterative process.

t-t: The preprocessing and solution time, i.e., the sum of the time required for preconditioning construction.

n-iter: The number of iterations required to find an approximate solution in each iteration.

If GMRES or CG fails to converge within 3000 iterations, it is denoted with "-".

\subsection{Example 1}
In this subsection, we will analyze the various factors that impact the effectiveness of PSLR preconditioner. We employ the PSLR-GMRES method to solve a saddle point matrix. Furthermore, we thoroughly investigate and analyze the factors that influence the PSLR preconditioner. Additionally, the time unit in the numerical example is in seconds.

In modern financial research, the arbitrage pricing theory models the expected return of a risky asset as a linear function of the asset's sensitivity to a set of factors, commonly referred to as a multifactor model\cite{ref4}. This model can be expressed as

 \begin{equation}
	\begin{bmatrix}
		A^TA&B^T\\-B&I
	\end{bmatrix}\begin{bmatrix}x\\y\end{bmatrix}=\begin{bmatrix}f\\g\end{bmatrix},\end{equation}
where the matrix has an order of 256, and the number of iteration steps of the Arnoldi algorithm is fixed at $r_k=15$. $\mathcal{A}$  is split into four tiled matrices of the same order.  We conduct testing using the following numerical matrix.
$$A^TA=
\begin{bmatrix}
	4&1&0&\cdots&0\\
	1&4&1&\cdots&0\\
	0&1&4&\cdots&0\\
	\cdots&\cdots&\cdots&\cdots&\cdots\\
	0&0&0&\cdots&4\\
\end{bmatrix}_{128\times128},
B^T=\begin{bmatrix}
	0&0&0&\cdots&0\\
	0&0&0&\cdots&0\\
	0&0&0&\cdots&0\\
	\cdots&\cdots&\cdots&\cdots&\cdots\\
	1&0&0&\cdots&0\\
\end{bmatrix}_{128\times128}.$$
And they forms a saddle point matrix. The right-hand vector $b$ is a random vector.

\subsubsection{The Impact of Different Types of Initial Values}

Expanding the five terms in a power series reveals that the impact of different types of initial values on solution efficiency varies. As shown in the Table \ref{tab1}, comparing the preprocessed initial value with the zero initial value and the random initial value, we conclude that although the three initial values have the same number of iteration steps, the preprocessed initial value has a shorter iteration time than the zero initial value and the random initial value.

\begin{table}[H]
	\begin{center}
		\caption{Comparison chart of different initial values}
  \label{tab1}
		\begin{tabular}{c|c|c|c}
			\hline
			Different initial values	&  n-iter& i-t&error   \\
			\hline
			Pre	& 11&0.000882 &9.322617e-07     \\
			
			Zero	&11&0.001128 &3.210799e-07 \\
			
			Random	& 12&  0.001770&2.947776e-07     \\
			\hline
		\end{tabular}
	\end{center}
\end{table}

"Pre" refers to the initial value obtained through the preprocessing process described earlier. "Zero" indicates that the initial value is zero. "Random" indicates that the initial value is a random vector of the same dimensions as the saddle point matrix. From the  Table \ref{tab1}, we can see that the PSLR method has certain advantages in terms of the number of iterations and the solution time.

\subsubsection{The Influence of Different Orders on Solution Efficiency}
To further investigate the saddle point system, we conducted tests of various orders and obtained the following results, as shown in the Table 
 \ref{tab2}.

\begin{table}[H]
	\begin{center}
		\caption{Comparison chart of different orders}
  \label{tab2}
		\begin{tabular}{c|c|c|c|c|c}
			\hline
			Order&128  &256&512  &1024 & 2048   \\
			\hline
			o-t&0.000185 &0.000255  &0.000696&0.002111 &0.010115  \\
			
			p-t	&0.014472  &0.028598& 0.113125& 0.693792 &  5.259069   \\
			
			i-t	&0.038024&0.189584& 1.053674 & 7.048978   &   46.408650    \\
			\hline
		\end{tabular}
	\end{center}
\end{table}

Based on the Table \ref{tab2} above, we can observe that the solution rate (number of iteration steps and CPU time) of the saddle point system increases with the increase in order. To better illustrate this, we have plotted a line chart as shown Fig. \ref{fig3}.

\begin{figure}[H]
	\centering
	\includegraphics[width=0.7\linewidth]{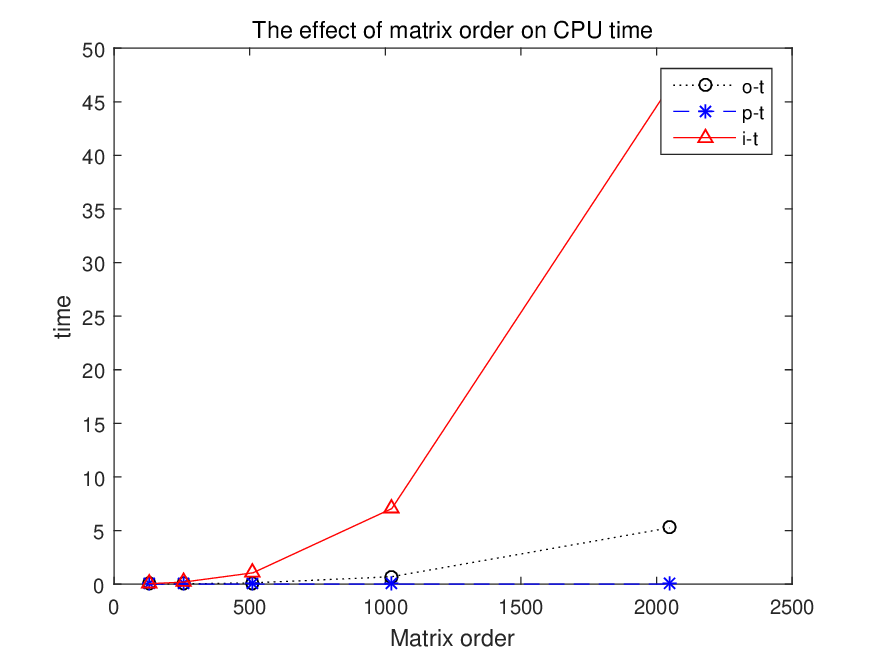}
	\caption{The effect of different orders on CPU time}
	\label{fig,The influence of matrix order on solving saddle point matrices}
 \label{fig3}
\end{figure}

\subsection{Example 2}

Consider a Stokes equation that satisfies the velocity vector $u$ and the pressure field $p$ in $\Omega$,\begin{equation}\label{eq17}\begin{cases}
		\xi u-v\Delta u+\nabla p=f \;\;in\;\Omega\\
		\nabla\cdot u=0 \;\;in \;\Omega
	\end{cases}.\end{equation}
We aim to determine the appropriate boundary condition on $\partial \Omega$ in Equation \eqref{eq17}, where $\Omega$ is a bounded domain in $\mathbb{R}^2$ or $\mathbb{R}^3$. The parameter $f$ represents a given force, and $v>0$ (viscosity) and $\xi \leq 0$ are additional parameters. When $\xi = 0$, the problem corresponds to the classical stationary Stokes problem.

We discretize the problem using both the standard finite difference method and the node finite element method, resulting in a linear system\begin{equation}\begin{bmatrix}
		A^TA&B^T\\B&C
	\end{bmatrix}\begin{bmatrix}
		u\\p
	\end{bmatrix}=\begin{bmatrix}
		b_u\\b_p
	\end{bmatrix}.\end{equation}
The matrices $A^TA$ and $C$ are symmetric positive definite, while $B$ is a sparse matrix. The right-hand side vector is randomly generated.

To better understand the saddle point system, we investigate the influence of various factors on the solution rate of the system through controlled variables.

\subsubsection{The Impact of Different Preprocessing $(A^TA)^{-1}$}
To solve $S^{-1}$, we employ the PSLR preprocessing method to expand $S^{-1}$ into a power series. However, for intermediate processing of $(A^TA)^{-1}$, we utilize alternative methods such as

Inversion using incomplete LU decomposition and LU decomposition of matrices \begin{equation}A^{-1}=U^{-1}L^{-1}.\end{equation}

Inversion using incomplete Cholesky decomposition and Cholesky decomposition of matrices \begin{equation}A^{-1}=(LL^T)^{-1}=(L^T)^{-1}L^{-1}.\end{equation}

We tested these methods using a symmetric matrix "494\_bus", and the results are shown in the following Table \ref{tab3}.

\begin{table}[H]
	\begin{center}
		\caption{Comparison chart of different inverse methods}
  \label{tab3}
		\begin{tabular}{c|c|c|c|c}
			\hline
			494\_bus &LU inverse &ILU inverse &IC inverse & Cholesky inverse \\
			\hline
			i-t &0.132069&  0.099848&0.092171&0.129752\\
			
			n-iter & 593 &517&511&594 \\
			
			error&9.765600e-07&9.302003e-07&9.695692e-07&9.990874e-07\\
			\hline
		\end{tabular}
	\end{center}
\end{table}

According to the results presented in the preceding Table \ref{tab3}, it has been observed that the IC approach exhibits the highest success rate in determining $(A^TA)^{-1}$. Consequently, If it is a symmetric matrix, we will employ the IC method for computing $(A^TA)^{-1}$ in the subsequent analysis.

\subsubsection{Effect of Matrix Ordering on Solution Rate}

To better understand the factors influencing the solution rate of the saddle point system, we tested the effect of matrix ordering. The results of the effect of matrix ordering on the solution rate are shown in the following Table \ref{tab4}.

\begin{table}[H]
	\begin{center}
		\caption{Comparison of the impact of whether the matrix is ordering or not}
   \label{tab4}
		\begin{tabular}{c|c|c|c}
			\hline
			Matrix name	& sort & i-t &  n-iter\\
			\hline
			\multicolumn{1}{c|}{\multirow{2}{*}{494\_bus}}	&yes &  0.358443 &324  \\
			
			&  no& 0.409336  &341  \\
			\hline
			\multicolumn{1}{c|}{\multirow{2}{*}{mcfe.mtx}}	&yes &8.112854  &1417  \\
			
			&  no& 8.512631   &1441  \\
			\hline
			\multicolumn{1}{c|}{\multirow{2}{*}{fs\_541\_3.mtx}}	&yes &0.491955   &396  \\
			
			&  no& 0.837130   &513 \\
			\hline	
		\end{tabular}
	\end{center}
\end{table}

We conducted tests on different matrices and found that matrix ordering affects the solution rate of the saddle point system. When the matrix is ordered, the number of iterations and solution time are reduced, resulting in an improved solution rate for the saddle point system.

\subsubsection{Effect of the Number of Power Series Terms $m$}

As mentioned earlier, the number of terms used in power series expansion is also an important factor. We studied this factor by solving the saddle point problem named "494\_bus", with $r_k=15$. The sorting time (o-t) is 0.000107 seconds.
The iteration counts and CPU times for different values of $m$ are shown in the Table \ref{tab5}. As $m$ increases from 0 to 5, the number of iterations varies. This can be attributed to the improved clustering of spectra of preconditioned Schur complement $m$ as the number of terms used in power series expansions increases. However, the time to build the PSLR preregulator has increased slightly. As $m$ increases from 0 to some normal number, the number of iterations decreases significantly and then decreases only slightly. We expect the iteration time to decrease first and then increase, which is verified by the numerical results in the Table \ref{tab5}. As $m$ increases from 0 to 2, the iteration time first decreases from 155 seconds to 0.39 seconds, and then as $m$ continues to increase, the iteration time first increases to 0.47 seconds and then decreases to 0.40 seconds.

\begin{table}[H]
	\begin{center}
		\caption{Effect of items on solution efficiency}
  \label{tab5}
		\begin{tabular}{c|c|c|c|c}
			\hline
			m	&  n-iter& i-t & p-t &t-t \\
			\hline
			0	&-& 155.194490&0.023766&150.218256 \\
			
			1	&320& 0.397969&0.020083 & 0.598799  \\
			
			2	&317& 0.398314 & 0.019542&0.417856\\
			
			3	&327& 0.402742&0.028274&0.431016	\\
			
			4	&323& 0.470357 &0.032013& 0.50237\\
			
			5	&324&  0.400080 &0.020267&0.420347\\
			\hline
		\end{tabular}
	\end{center}
\end{table}

The specific running time change trend is shown in the Fig. \ref{fig,untitled}.

\begin{figure}[H]
	\centering
	\includegraphics[width=0.7\linewidth]{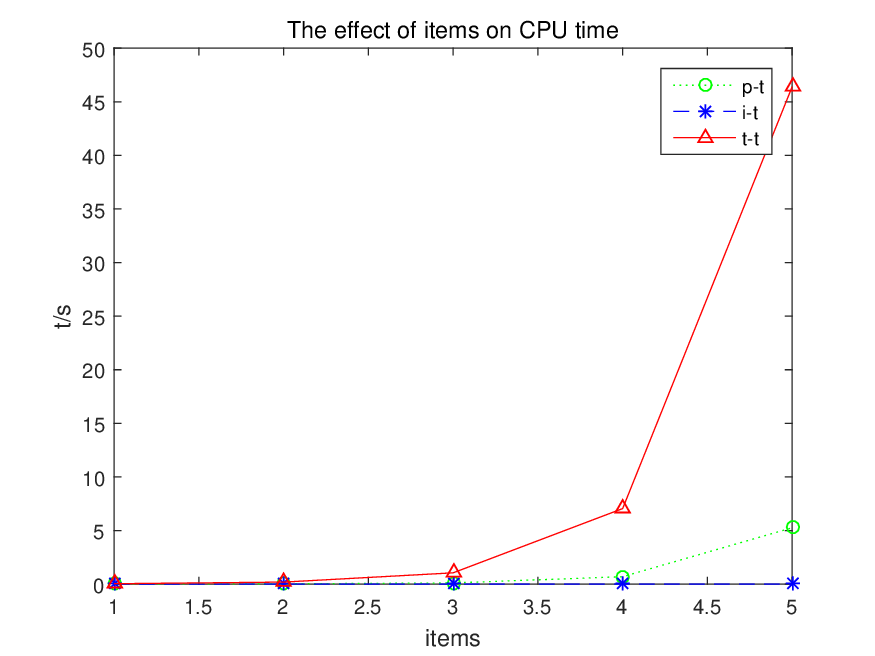}
	\caption{Influence of different power series expansion terms $m$ on saddle point matrix solving}
	\label{fig,untitled}
\end{figure}

\subsubsection{Impact of Different Solution Method}
We compared the convergence rates of four methods, namely PSLR-GMRES, Pinv, CG, and ADI, for solving saddle point systems. The results are as follows.

\textbf{1.Comparison of CG method and PSLR-GMRES method.}
We investigated the impact of different methods for solving the saddle point system and obtained the results shown in the following Table \ref{tab6}. The letters "IC", "ILU", and "LU" in the matrix name refer to the method used for solving the inverse of the matrix. Different matrices correspond to different methods for solving the inverse of $B$ in the tiled matrix.

\begin{table}[H]
	\begin{center}
		\caption{Comparison table of the effects of PSLR method and CG method}
  \label{tab6}
		\begin{tabular}{c|c|c|c|c}
			\hline
			\multicolumn{1}{c|}{\multirow{2}{*}{Matrix name}}&	\multicolumn{2}{c}{PSLR-GMRES}	&\multicolumn{2}{|c}{CG}  \\
			\cline{2-5}
			&  n-iter& i-t  &n-iter&i-t \\
			\hline
			1138\_bus.mtx+IC& 581 & 1.634505  & 2715  &0.268293   \\
			
			685\_bus.mtx+IC&333&0.401070 &594&0.067242 \\
			
			494\_bus.mtx+IC&334&0.330287  &1448&0.079305\\
			
			bcsstk03.mtx+ILU&334&0.219232 &590&0.018062 \\
			
			bcspwr02.mtx+LU&92&0.025676&-&0.068686  \\
			\hline
		\end{tabular}
	\end{center}
\end{table}

It can be seen from the above Table \ref{tab6} that the solution time and solution speed after processing with PSLR preprocessor are greater than the solution time and solution speed of the conjugate gradient method, which shows the superiority of PSLR to a certain extent.

\textbf{2.Comparison of Pinv method and CG method.}
To demonstrate the applicability of the proposed method to general systems, we tested several sparse linear systems obtained from the Matrix Market database, with a brief description provided in the Table \ref{tab7}. We fixed $m = 5$ and $r_k = 15$ for all experiments. The numerical results are shown in the Table \ref{tab7}, indicating that the PSLR-GMRES method converges for all test cases. Moreover, the number of iterations and iteration time are smaller than those of the CG method.

\begin{table}[H]
	\begin{center}
		\caption{Comparison table of the effects of Pinv method and CG method}
  \label{tab7}
		\begin{tabular}{c|c|c|c|c}
			\hline
			\multicolumn{1}{c|}{\multirow{2}{*}{Matrix name}}&	\multicolumn{2}{c}{Pinv}	&\multicolumn{2}{|c}{CG}  \\
			\cline{2-5}
			&  n-iter& i-t  &n-iter&i-t \\
			\hline
			bcsstk03.mtx& 562  &0.835204  & 590  & 0.019754  \\
			
			494\_bus.mtx& 1138& 4.205660&1445&0.067710 \\
			\hline
		\end{tabular}
	\end{center}
\end{table}

The Table \ref{tab7}  clearly shows that the Pinv method outperforms CG method in terms of iteration time and number of iterations for certain matrices.

\textbf{3.Comparison of ADI and PSLR-GMRES method.}
We can get the ADI iteration format
\begin{equation}
	\begin{cases}
		(\mathcal{H}+\alpha\mathcal{I})x^{k+\frac{1}{2}}=(\alpha\mathcal{I}-\mathcal{S})x^k+b\\
		(\mathcal{S}+\alpha\mathcal{I})x^{k+1}=(\alpha\mathcal{I}-\mathcal{H})x^{k+\frac{1}{2}}+b\\
	\end{cases}.\end{equation}

In the initial half of the iterations, we employ the preprocessing conjugate gradient method (PCG) for the first iteration, while for the latter half, we utilize the GMRES iteration for the second iteration.

We implemented this iteration format using MATLAB and set the maximum number of iterations to 300, terminating the iteration when the error is less than $10^{-6}$. A comparison with PSLR-GMRES is shown in the following Table \ref{tab8}. We tested the three-diagonal matrix of order 128, the five-diagonal matrix of order 128, and the seven-diagonal matrix of order 256. The three-, five-, and seven-diagonal matrices are presented below.

$$
	\begin{bmatrix}
		6&2&0&\cdots&0\\
		2&6&2&\cdots&0\\
		0&2&6&\cdots&0\\
		0&0&2&\cdots&0\\
  \cdots&\cdots&\cdots&\cdots&\cdots\\
		0&0&0&\cdots&6\\
	\end{bmatrix},\begin{bmatrix}
		6&2&1&\cdots&0\\
		2&6&2&\cdots&0\\
		1&2&6&\cdots&0\\
		0&1&2&\cdots&0\\
  \cdots&\cdots&\cdots&\cdots&\cdots\\
		0&0&0&\cdots&6\\
	\end{bmatrix},\\
	\begin{bmatrix}
		6&2&1&0.5&\cdots&0\\
		2&6&2&1&\cdots&0\\
		1&2&6&2&\cdots&0\\
		0.5&1&2&6\cdots&0&0\\
		\cdots&\cdots&\cdots&\cdots&\cdots\\
		0&0&0&0&\cdots&6\\
	\end{bmatrix}.$$The parameters used in the ADI method are $\alpha=1.5$, and the right-hand vector is an all-one vector with the dimension of the matrix.

\begin{table}[H]
	\begin{center}
		\caption{Comparison table of the effects of PSLR-GMRES method and ADI method}
  \label{tab8}
		\begin{tabular}{c|c|c|c|c|c|c}
			\hline
			\multicolumn{1}{c}{\multirow{2}{*}{Matrix}}&\multicolumn{3}{|c}{PSLR-GMRES}	&\multicolumn{3}{|c}{ADI}  \\
			\cline{2-7}
			&i-t	& n-iter& error  &i-t&n-iter&error \\
			\hline
			Tridiagonal matrix&0.011&14&8.17e-06& 0.092 &18&8.76e-07\\
			
			Five diagonal matrix&0.029 &14&9.86e-07&0.054&12&9.12e-07\\
			
			Seven-diagonal matrix&0.002 &13&4.31e-07&0.093 &9&7.66e-07\\
			\hline
		\end{tabular}
	\end{center}
\end{table}
\newpage
For a better description, We plotted an error figure \ref{fig,Error trend graph}.
\begin{figure}[H]
	\centering
	\includegraphics[width=0.7\linewidth]{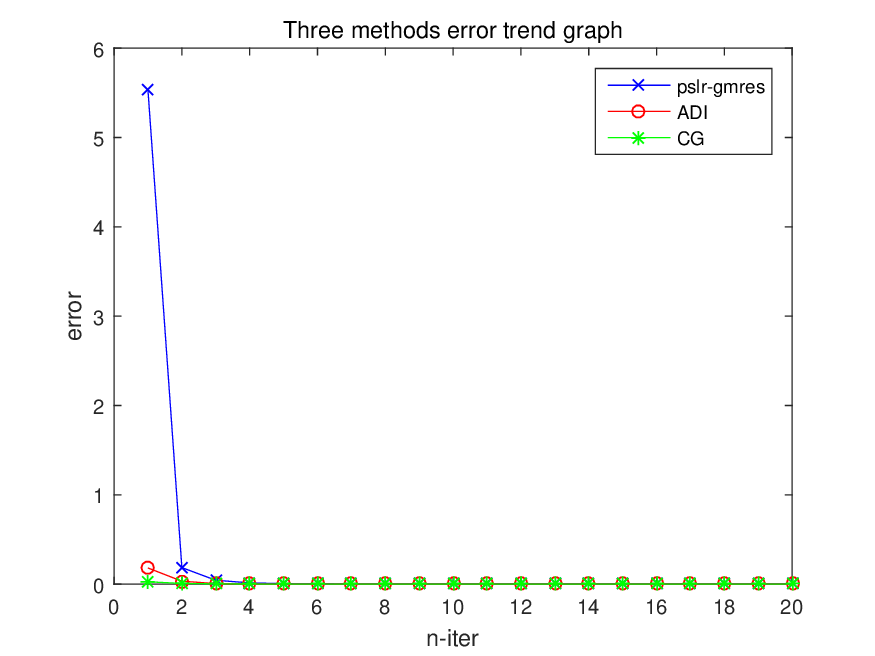}
	\caption{Three methods error trend chart}
	\label{fig,Error trend graph}
\end{figure}

From the Tables \ref{tab8} and Fig. \ref{fig,Error trend graph}, for triangular diagonal matrices, we can see that the PSLR-GMRES method has a better convergence rate than the ADI method in solving this equation. But with five-diagonal and seven-diagonal matrices, the PSLR-GMRES method is inferior to the ADI method.
\subsubsection{Comparison of different pretreatment methods}
We compared the PSLR-GMRES method with PCG method and Jacobi Preconditioning, and the results are shown in the Table \ref{tab20}, Table \ref{tab21} below. From the Table \ref{tab20}, Table \ref{tab21}, we can see that PSLR-GMRES has certain advantages in solving saddle point systems. The tridiagonal matrix involved in the numerical example is: $$\begin{bmatrix}
	2&0.5&0&\cdots&0\\
	-1&2&0.5&\cdots&0\\
	0&-1&2&\cdots&0\\
	\vdots&\vdots&\vdots&\cdots&\vdots\\
	0&0&0&\cdots&2\\
	\end{bmatrix}_{128\times128}.$$

\begin{table}[H]
	\begin{center}
		\caption{Comparison table of the effects of PSLR method and PCG method}
        	\label{tab20}
		\begin{tabular}{c|c|c|c|c|c|c}
			\hline
			\multicolumn{1}{c|}{\multirow{2}{*}{Matrix name}}&	\multicolumn{3}{c}{PSLR-GMRES}	&\multicolumn{3}{|c}{PCG} \\
			\cline{2-7}
			 &  n-iter& error&i-t&n-iter&error&i-t \\
			\hline
			
			494\_bus & 338&9.137512e-07&0.343641 &  3000&3.320152e-02 &0.080933\\
			
			tridiagonal matrix&17&6.233525e-07 &0.022151  &20&5.160190e-07&0.011246\\
			
			\hline
		\end{tabular}
	\end{center}
\end{table}
The preprocessing matrix of 494\_bus is $A$'s incomplete decomposition factor, and the preprocessing matrix of the other matrix is the matrix composed of its diagonal elements.

\begin{table}[H]
	\begin{center}
		\caption{Comparison table of the effects of PSLR method and Jacobi-pre method}
        	\label{tab21}
		\begin{tabular}{c|c|c|c|c|c|c}
			\hline
			\multicolumn{1}{c|}{\multirow{2}{*}{Matrix name}}&	\multicolumn{3}{c}{PSLR-GMRES}	&\multicolumn{3}{|c}{Jacobi-pre}  \\
			\cline{2-7}
			 &  n-iter& error&i-t &n-iter&error&i-t\\
			\hline
			
			tridiagonal matrix&17&6.233525e-07 &0.022151 &50&9.748978e-07 &0.004636\\
			
			dwb512&10&5.119408e-07 &0.001318 &11&5.900197e-07 &0.241068\\

			\hline
		\end{tabular}
	\end{center}
\end{table}
where Jacobi-pre stands for Jacobi Preconditioning, the preprocessing matrix in Jacobi-pre is the matrix composed of the reciprocal of the diagonal element of $A$.
\section{Conclusions}
\label{sec:conclusions}

The work of this paper is mainly to preprocess the saddle point system, and use the power series expansion to approximate the inverse matrix of the Schur complement matrix, so as to avoid the situation that the inverse matrix does not exist. At the same time, when solving the matrix equation system, the GMRES method can improve the rate of solving the saddle point system. Compared with other classical methods, our method needs to depend on specific problems and has certain limitations, so it needs to be further improved.

\section{Statements and Declarations}
No potential conflict of interest was reported by the authors.


\end{document}